\documentclass[12pt,reqno]{amsart}
\usepackage{latexsym,amsmath}
\usepackage[left=3cm,top=2.5cm,right=3cm,bottom=2.5cm]{geometry}
\usepackage[usenames,dvipsnames]{color}
\usepackage{amsthm}
\usepackage{amssymb}
\usepackage{epsfig}
\usepackage{mathrsfs}
\usepackage{verbatim}
\usepackage{float} 

\usepackage[algoruled,english,linesnumbered]{algorithm2e}
\usepackage{multido}

\usepackage{amsmath}

\newtheorem{theorem}{Theorem}[section]
\newtheorem{lemma}[theorem]{Lemma}

\newtheorem{proposition}[theorem]{Proposition}

\newtheorem{conjecture}[theorem]{Conjecture}

\theoremstyle{definition}

\def\output{\mathscr{E}}
\newcommand{\C}{\mathcal{C}}
\begin{document}

  \title{Asymptotic bounds on total domination in regular graphs} \author{Carlos Hoppen}
  \address{Instituto de Matem\'{a}tica e Estat\'{i}stica\\ Universidade Federal do Rio Grande do Sul\\
    Porto Alegre, Brazil}\email{choppen@ufrgs.br}
  \author{Giovane Mansan}
  \address{Instituto de Matem\'{a}tica e Estat\'{i}stica\\Universidade Federal do Rio Grande do Sul\\
    Porto Alegre, Brazil}  
\email{giovanemansan@gmail.com}

\thanks{A preliminary version of this paper appeared as an extended abstract in the Proceedings of the X Latin and American Algorithms, Graphs and Optimization Symposium (LAGOS'19)~\cite{gio3}.}

\begin{abstract}    
We obtain new upper bounds on the size of a minimum total dominating set for random regular graphs and for regular graphs with large girth. In particular, they imply that an upper bound conjectured by Thomass\'{e} and Yeo~\cite{Tom} holds asymptotically almost surely for 5-regular graphs and holds for all 5-regular graphs with sufficiently large girth. Our bounds are obtained through the analysis of a local algorithm using a method due to Hoppen and Wormald~\cite{wor}. 
\end{abstract}

\maketitle

\section{Introduction and Main Results}\label{intro}
This paper is about total dominating sets in graphs. As usual, a graph $G=(V,E)$ consists of a \emph{vertex set} $V$ and of an \emph{edge set} $E \subseteq \{\{u,v\} \colon u,v \in V, u \neq v\}$. Our notation and terminology is standard, we refer the reader to~\cite{alon}.  

There are many parameters related with the general notion of domination in graphs. The most studied is the \emph{domination number} $\gamma(G)$ of a graph $G=(V,E)$. A set $S \subseteq V$ is a \emph{dominating set} of $G$ if every vertex in $V \setminus S$ is adjacent to some vertex in $S$. The \emph{domination number} is the minimum size of a dominating set of $G$, that is,
$$\gamma(G)=\min\{|S| \colon S \textrm{ is a dominating set of }G\}.$$
Total domination is a related notion that has been introduced by Cockayne, Dawes and Hedetniemi~\cite{dawes}. A \emph{total dominating set} $S  \subseteq V$ is a set such that every vertex $v \in V$ is adjacent to a vertex in $S$, so that every vertex in a total dominating set must have a neighbor in this set. In particular, a total dominating set if and only if it does not have isolated vertices. Naturally, the \emph{total domination number} $\gamma_t(G)$ of $G$ is defined as 
\[\gamma_t(G)=\min\{|S| \colon S \textrm{ is a total dominating set of }G\}.\]
Note that every total dominating set is a dominating set, and that every dominating set $S$ may be turned into a total dominating set by adding a neighbor of each vertex in $S$ to the set, we have 
$\gamma(G) \leq \gamma_t(G) \leq 2 \gamma(G)$ for any graph $G$ with no isolated vertices. Both inequalities are tight, as illustrated by complete bipartite graphs and complete graphs, respectively.

Computing the value of $\gamma(G)$ appears on Karp's seminal list of NP-complete problems~\cite{karp}. Pfaff, Laskar and Hedetniemi~\cite{PLH83} proved that computing $\gamma_t(G)$ is also NP-complete. For results and references about domination and total domination, we refer to Haynes, Hedetniemi and Slater~\cite{HHS98} and to Henning and Yeo~\cite{hen}, respectively. 

A large number of upper and lower bounds have been proposed for the size of a minimum total dominating set in an $n$-vertex graph $G$. Since the addition of a vertex to a set may only dominate its neighbors, it is clear that $\gamma_t(G) \geq n/\Delta(G)$, where $\Delta(G)$ denotes the maximum degree of $G$. On the other hand, Henning and Yeo~\cite[Theorem 5.1]{hen} proved that $\gamma_t(G) \leq \left(\frac{1+\ln(\delta)}{\delta}\right)n$, where $\delta(G) \geq 1$ is the minimum degree of $G$. A natural setting for comparing upper and lower bounds of this type are \emph{$d$-regular} graphs, namely graphs where every vertex is incident with $d$ edges, so that $\delta(G)=\Delta(G)=d$ (we shall always assume that $nd$ is even). 

In general, the size of a minimum total dominating set may still vary considerably among $n$-vertex $d$-regular graphs. For instance, if $G$ is a collection of disjoint complete bipartite graphs $K_{d,d}$, we have $\gamma_t(G)=n/d$, as every component is totally dominated by one vertex of each side of the bipartition. This shows that the trivial lower bound on $\gamma_t(G)$ mentioned above is sharp for all $d$. On the other hand, if $G$ is a collection of disjoint complete graphs $K_{d+1}$, we have $\gamma_t(G)=2n/(d+1)$, which is substantially larger. Table~\ref{tabela} gives upper bounds $\Gamma_0(d)$ on the proportion of vertices in a minimum total dominating set in a $d$-regular graph $G$ for some values of $d$ (actually, these bounds have been obtained for $ \delta(G)\geq d$). As it turns out, the upper bounds in the table are sharp for $d$-regular graphs for $d \in \{2,3,4\}$. For $n$-vertex graphs $G$ with $\delta(G) \geq 5$, the best known upper bound $\gamma_t(G) \leq \frac{17}{44} n$ is due to Dorfling and Henning~\cite{Dorf}. Thomass\'{e} and Yeo~\cite{Tom} conjecture that the following improvement is possible.
\begin{conjecture}\label{conj_TY}
Every $n$-vertex graph $G$ with minimum degree $\delta(G) \geq 5$ satisfies
$$\displaystyle{\gamma_t(G) \leq \frac{4}{11} n.}$$
\end{conjecture} 

To investigate the behavior of a graph-theoretical parameter on $d$-regular graphs avoiding the influence of particular substructures, two traditional ways are to consider its \emph{typical value}, namely its value for a randomly chosen $d$-regular graph, and to consider its value on graphs with \emph{large girth}, where the \emph{girth} of a graph is the length of a shortest cycle in the graph. Properties of random regular graphs have been intensively studied (see~\cite{Wormald} for a survey of results in this direction). The effect of the girth on the value of graph parameters has also been widely studied, two classical references involving the chromatic number are Erd\H{o}s~\cite{erd59} and Gr\"{o}tzsch~\cite{grot}.   

Regarding the effect of large girth on the total domination number of a $d$-regular graph $G$, Henning and Yeo~\cite{henn} showed that, if $G$ is an $n$-vertex graph with $\delta(G)\geq 2$ and  girth $g \geq 3$, then  $\gamma_t(G)\leq \left(\frac{1}{2}+\frac{1}{g}\right)n.$
In particular, this shows that the trivial lower bound is asymptotically optimal for $2$-regular graphs as $g \rightarrow \infty$. (This fact can also be proved directly by looking at minimum total dominating sets of long cycles.) We study this parameter for general $d$. Precisely, for $d\geq 2$ and $g_0\geq 3$, let
\begin{equation}\label{def_par_girth}
\gamma_t^g(d,g_0)=\sup\{ \gamma_t(G)/|V(G)| \colon \textrm{$G$ is $d$-regular with girth $g\geq g_0$} \},
\end{equation}
that is, $\gamma_t^g(d,g_0)$ is the smallest possible upper bound on $d$-regular graphs with girth at least $g_0$. This produces a monotone non-increasing sequence as $g_0$ increases, and we consider the parameter
\begin{equation}\label{def_par_lg}
\gamma_t^g(d,\infty)=\lim_{g_0\rightarrow\infty}\gamma_t^g(d,g_0).
\end{equation}
The result of Henning and Yeo implies that $\gamma_t^g(2,\infty)=1/2$.

The following is the main result of this paper.
\begin{theorem} \label{thm_main}
For any $d \geq 3$ and $\delta>0$, there exists $g_0$ such that any $d$-regular graph $G$ with girth $g \geq g_0$ satisfies
$$ \gamma_t(G)/|V(G)| \leq q(x^*)+\delta,$$
where $z_0(x)$ and $q(x)$ are solutions to the initial value problem (\ref{eq:53}) and $x^*=\inf\{ x>0:z_0(x)=0\}$.
\end{theorem}

The system of differential equations mentioned in the statement of the theorem arises naturally as we analyse the algorithm described in Section~\ref{sec:2}. We were not able to solve the initial value problem (\ref{eq:53}) analytically, and Table~\ref{tabela} provides numerical upper bounds $\Gamma^g_d$ (where the fourth decimal place has been rounded up) on the value of $q(x^*)$ in Theorem~\ref{thm_main} for a few values of $d$.
\begin{table}
\begin{center}
\begin{tabular}{|c|c|c|c|}
\hline
  $d$ &   $\Gamma_0(d)$ & Source &  $\Gamma^g_d$\\
\hline
2 &     $2/3$  & See footnote\footnote{This bound may be improved to $4/7$ if the components of $G$ are not in a particular family of six small graphs \cite{henning00,sun}.} &  $1/2$~\cite{henn} \\
\hline
3 &     $1/2$  & \cite{Arch} (2004)   & 0.4762  \\
\hline
4 &    $3/7\simeq 0.4285$  & \cite{Tom} (2007) & 0.4055  \\
\hline
5 &     $17/44\simeq 0.3863$& \cite{Dorf} (2015)  & 0.3572 \\
\hline
8 &   $0.3849$   &  \cite[Theorem 5.1]{hen} & 0.2703  \\
\hline
\end{tabular}
\end{center}
\caption{Deterministic upper bounds $\Gamma_0$ on the size of a minimum total dominating set in a $d$-regular graph $G$, the corresponding references and numerical approximations of the upper bound of Theorem~\ref{thm_main}.}
\label{tabela}
\end{table}
In particular, since $0.3572 < \frac{4}{11} \simeq 0.3636$, this shows that Conjecture~\ref{conj_TY} must hold for all $5$-regular graphs with sufficiently large girth.

Next consider the typical value of the total domination number on a large $d$-regular graph. To this end, let $\mathbb{G}_{n,d}$ be the set of (labelled) $n$-vertex $d$-regular graphs and, for an integer $d \geq 2$ and a constant $\varepsilon>0$, consider
\begin{equation}
\gamma_t^R(d,\varepsilon) = \inf_{\substack{\mathcal{A}\subseteq \mathbb{G}_{n,d},n\in\mathbb{N}, \\ |\mathcal{A}|\geq (1-\varepsilon)|\mathbb{G}_{n,d}|}} \sup \left\{ \frac{\gamma(G)}{n} \colon G\in \mathcal{A} \right\}.
\end{equation}
Note that, for fixed $d$, $\gamma_t^R(d,\varepsilon)$ is bounded and increases as $\varepsilon$ decreases, so that the following limit is well-defined:
\begin{equation}
\gamma_t^R(d)= \lim_{\varepsilon\rightarrow 0^+}\gamma_t^R(d,\varepsilon).
\end{equation}
Let $\mathcal{G}_{n,d}$ denote the probability space with sample space $\mathbb{G}_{n,d}$ and uniform probability distribution. In the language of probability, finding an upper bound $\Gamma^r_d$ on $\gamma_t^R(d)$ means that a random $d$-regular graph \emph{asymptotically almost surely} (a.a.s.) has a minimum total dominating set of size at most $\Gamma^r_d$. 

A well-known construction, which uses the fact that random $d$-regular graphs a.a.s.\ have a small number of cycles of bounded length~\cite{bolob,wor2}, allows one to prove that 
\begin{equation}\label{eq_lg_rrg}
\gamma_t^R(d) \leq \gamma_t^g(d,\infty), 
\end{equation}
so that any deterministic upper bound on the total domination number of $d$-regular graphs with large girth gives us an upper bound on the total domination number of a typical $d$-regular graph. In fact, the connection between the behavior of graph parameters for graphs with large girth and for random regular graphs given in~\eqref{eq_lg_rrg} in the context of total domination actually holds for many different parameters, and it is a significant open question whether inequalities such as~\eqref{eq_lg_rrg} hold with equality (see Backhausz and Szegedy~\cite{BS2018} for a detailed description of problems in this line of research). 

Wormald and the first author~\cite{wor} proved that an upper bound on $\gamma_t^R(d)$ also implies an upper bound on $\gamma_t^g(d,\infty)$ provided that it is obtained through the analysis of a \emph{local algorithm}, as described in their paper. (Again, the previous sentence would hold for a host of parameters other than total domination.) This result by Hoppen and Wormald (\cite[Theorem 7.1]{wor}) plays a fundamental role in this paper, as it allows us to derive Theorem~\ref{thm_main} from the proof a theorem about random regular graphs.
\begin{theorem} \label{teo:52}
For any $d \geq 3$ and $\delta>0$, a random graph $G \in \mathcal{G}_{n,d}$ asymptotically almost surely contains a total dominating set $D_T\subseteq V(G)$ such that
\[|D_T|\leq n\left(q(x^*)+\delta\right),\]
where $z_0(x)$ and $q(x)$ are solutions to the initial value problem (\ref{eq:53}) and $x^*=\inf\{ x>0:z_0(x)=0\}$.
\end{theorem}
The proof of Theorem~\ref{teo:52} uses a powerful method due to Wormald~\cite{wormald1}, known as the differential equation method. It analyses the performance of a specific local algorithm that produces a total dominating set in an input graph $G$ when this algorithm is applied to a random regular graph $G \in \mathcal{G}_{n,d}$. The differential equation method is a concentration-type result that has been very successful in the analysis of random processes. In the particular case of random regular graphs, it has already been used to study parameters related with domination, see~\cite{duc,worth}, and results for graphs with large girth using the general approach described above have also been proved in~\cite{HW}. 

We should also mention that the ability of local algorithms to approximate the value of graph parameters for graphs with large girth has attracted a lot of attention. Gamarnik and Sudan~\cite{gam} showed that, for sufficiently large $d$, local algorithms cannot approximate the size of the largest independent set in a $d$-regular graph of large girth with an arbitrarily small multiplicative error. The approximation gap was improved by Rahman and Vir\'{a}g~\cite{RV}. Very recently, the same phenomenon was observed for max-cut problems~\cite{chen}. However, to the best of our knowledge, there are no results for domination parameters or for small values of $d$.

The remainder of the paper is structured as follows. In Section~\ref{sec:2}, we present our algorithm, while in Section~\ref{sec:2_1} we describe the setting in which the analysis is carried out. Section~\ref{sec:3} contains the proof of our main result.

\section{A heuristic to produce small total dominating sets}\label{sec:2}

Given a $d$-regular graph $G$, we may easily devise heuristics to produce small total dominating sets. For instance, start with the graph $G_0=G$ and a set $D_{0}=\emptyset$, which will be the total dominating set at the end of the heuristic. The construction proceeds by rounds that are labeled by a discrete parameter $t$. For each $t$, we produce $G_{t+1}$ and $D_{t+1}$ from $G_t$ and $D_t$, respectively, according to the following rules:
\begin{enumerate}
\item[$(1)$] Choose a vertex $v_t$ u.a.r.\ among all vertices of degree $d$ in $G_t$ and choose a vertex $u_t$ u.a.r.\ among the neighbors of $v_t$.

\item[$(2)$] If the degree of $u_t$ in $G_t$ is $d$, then delete $u_t$ and $v_t$ to produce $G_{t+1}$ and define $D_{t+1}=D_t \cup \{ v_t,u_t \}$.

\item[$(3)$] If the degree of $u_t$ in $G_t$ is not $d$, then delete $u_t$ to produce $G_{t+1}$ and define $D_{t+1}=D_t \cup \{ u_t \}$.
\end{enumerate}
Note that the set of vertices of degree $d$ in $G_t$ is precisely the set of vertices of $G_t$ that are not dominated by vertices in $D_t$. Moreover, all vertices added to $D_t$ to produce $D_{t+1}$ are dominated by a vertex that is already in $D_t$ or that is added to $D_{t+1}$. As a consequence, if this sequence of steps were performed until $G_t$ did not contain any vertices of degree $d$, then the set $D_t$ would be a total dominating set of $G$. This simple heuristic has been analysed in~\cite{gio3}. Although the asymptotic upper bounds provided are better than the deterministic upper bounds that hold for all $d$-regular graphs, they are not good enough, for instance, to prove that the upper bound  in Conjecture~\ref{conj_TY} holds for all $5$-regular graphs with sufficiently large girth (and that it holds a.a.s.\ for $5$-regular graphs). To obtain better results, we devise a heuristic that takes some additional information into account.

\begin{algorithm}[H]
\linespread{1.2}\selectfont
\SetAlgoRefName{$1(l)$}
\KwIn{An $n$-vertex $d$-regular graph $G$.}
\KwOut{A total dominating set $D$ of $G$.}
Set $t=0$, $G_0=G$ and $D_0=\emptyset$\;
\While{the number of vertices of degree $d$ in $G_t$ is at least $\epsilon n$}{
Choose a vertex $v_t$ u.a.r.\ among all vertices of degree $d$ in $G_t$\;
Choose a neighbor $u_t$ of $v_t$ u.a.r.\ \;
\uIf{$\deg_{G_t}(u_t)\neq d$}{$D_{t+1}=D_t \cup \{u_t\}$, $G_{t+1}=G_t-\{u_t\}$ and $t \leftarrow t+1$\;}\Else{
\uIf{$u_t$ has a neighbor of degree $d$ in $G_t$ other than $v_t$}{$D_{t+1}=D_t \cup \{u_t,v_t\}$, $G_{t+1}=G_t-\{u_t,v_t\}$ and $t \leftarrow t+1$\;}
\Else{\uIf{$v_t$ has a neighbor $w_t$ of degree $d$ other than $u_t$}{$D_{t+1}=D_t \cup \{w_t,v_t\}$, $G_{t+1}=G_t-\{w_t,v_t,u_t\}$ and $t \leftarrow t+1$;}\Else{choose u.a.r. a neighbor $v_t'$ of $v_t$ and a neighbor $u_t'$ of $u_t$\;
$D_{t+1}=D_t \cup \{u'_t,v'_t\}$, $G_{t+1}=G_t-\{u_t,u'_t,v_t,v'_t\}$ and $t \leftarrow t+1$\;}}}
}
Add a neighbor of each vertex of degree $d$ in $G_t$ to $D_t$ to produce $D$;
\caption{with parameters $n\geq 4$, $d\geq 3$ and $\epsilon>0$.}
\label{algo:totall}
\end{algorithm}

As in the previous heuristic, vertices with degree $d$ in $G_t$ are precisely the ones that have not been dominated up to round $t$. The difference here is that, when we add two vertices to the dominating set in a single round, we make more effort to choose vertices of degree $d$ that also dominate other vertices. This is a local improvement that will lead to better results. 

The methods that we use do not allow us to analyse such an algorithm after the first round $T$ such that the number of vertices of degree $d$ in $G_T$ falls below $\varepsilon n$, where $n=|V(G)|$ and $\varepsilon>0$ is a small constant. However, these vertices may be easily dominated with the addition of at most $\varepsilon n$ vertices to $D_T$. 

Before discussing how this algorithm may be analysed, we observe that it is a local deletion algorithm in the sense of~\cite{wor}, which allows us to derive Theorem~\ref{thm_main} from Theorem~\ref{teo:52} by applying  \cite[Theorem 7.1]{wor}. Starting with an input graph $G_0=G$ with colored vertices, a \emph{local deletion algorithm}  is an iterative algorithm which, at each round $t \geq 1$, randomly selects some vertices in the \emph{survival graph} $G_{t}$ according to a certain kind of rule and explores the neighborhoods of the selected vertices to create a new survival graph $G_{t+1}$, which is obtained by possibly recoloring or deleting some of the vertices of $G_t$. Let $D$ be a positive integer, called the \emph{depth} of the algorithm. The \emph{type} of a vertex $v$ in a colored graph is a pair given by its color and its degree. Assume that there are two sets of colors $\C$ and $\output$, which denote respectively the set of \emph{transient colors}, which are assigned to the vertices in the survival graph, and the set of \emph{output colors}, which are assigned to the vertices that are deleted from the survival graph. 

Initially, all vertices have the same transient color, which is called {\em neutral}. At each round $t \geq 1$, the algorithm produces a survival graph $G_{t+1}$ according to the following rules. First, the algorithm produces a subset $S_t \subset V(G_{t})$. For each $v \in S_t$, the algorithm \emph{explores} some vertices within distance $D$ of $v$ in $G_{t}$ according to some rules that will be specified below. Exploring a vertex means checking the type of some of its neighbors and possibly adding them to a \emph{query graph}, a colored graph that stores this information and allows one to explore further. Once a final query graph is obtained, there is a \emph{recoloring step} in which some vertices of $G_{t}$ are assigned new colors. Those that receive output colors are deleted from $G_{t}$ to produce the new survival graph $G_{t+1}$. Algorithm~\ref{algo:totall} may be viewed as having a single transient color, the neutral color, for all vertices in the survival graph and two output colors, one for vertices that have been added to $D$ and another for the other deleted vertices (lines 12 and 15 of the algorithm). In particular, the type of a vertex in the survival graph is simply given by its degree. This is known as a \emph{native local deletion algorithm} in~\cite{wor}.

In local deletion algorithms, selecting vertices at each round $t$ uses a \emph{selection rule}, which is a randomised function that, for a nonempty colored graph, produces a subset $S$ of $V$. It must have the property that any vertex $v$ lies in $S$ with a probability that is determined by its type. In Algorithm~\ref{algo:totall}, a single vertex $v_t$ is selected at each round $t$, always with degree $d$, and all vertices of degree $d$ are equally likely, so that this is satisfied. A local deletion algorithm is allowed to explore the neighborhood of each vertex $v_t$ that is selected at round $t$ up to distance $D$. This is also an iterative procedure, starting from the singleton $v_t$ and producing a query graph $Q(v_t)$. Every time a vertex is added to $Q(v_t)$, we may \emph{query} it, that is, ask about the type of some of its neighbours. Decisions about the type of next vertex to query or to add to $Q(v_t)$ (or about ending exploration) may be based on previous queries and may involve randomisation. However, if several vertices of the same type are candidates to be chosen in the same step, they must be chosen with the same probability \footnote{Actually, the description of this step in~\cite{wor} is more precise, but the current description is sufficient for the purposes of this paper.}. In Algorithm~\ref{algo:totall}, the query graph always includes $v_t$ and $u_t$, but may also include $w_t$, $v'_t$ and $u'_t$ depending on the outcome of queries. Exploration may go up to distance 2 of $v_t$ (in case $u_t'$ is chosen). Recall that, whenever the algorithm calls to select a neighbor with a given degree, it is chosen u.a.r.\ amongst all vertices with that degree, so that this condition is satisfied. 

The final ingredient for defining a local deletion algorithm is a \emph{recoloring rule}, which defines how the algorithm uses the information given in the exploration step to update the survival graph. All vertices in the query graph that are deleted from the survival graph must be assigned output colors, while the remaining vertices in the query graphs and their neighbors in $G_t$ may keep their color or be recolored with another transient color \footnote{In~\cite{wor} there are rules for recoloring, but they are not relevant in the context of a native local deletion algorithm where a single vertex is selected at each round.}. In Algorithm~\ref{algo:totall}, all vertices in the survival graph have a single transient color, so that recoloring is trivial and clearly satisfies these requirements. 

This allows us to conclude that Algorithm~\ref{algo:totall} is a local deletion algorithm, so that \cite[Theorem 7.1]{wor} allows us to derive Theorem~\ref{thm_main} directly for Theorem~\ref{teo:52}.

\section{Random regular graphs and the Differential Equation Method}\label{sec:2_1}

The previous section was devoted to introducing the heuristic that will be analysed in this paper and to showing, based on work in~\cite{wor}, that it suffices to analyse the performance of this heuristic in random regular graphs. This will be done using the well-known differential equation method, which is the subject of this section. As in many applications of this method to random regular graphs, instead of working directly with regular graphs, we use the approach of Bollob{\'a}s~\cite{bolob}, known as the \emph{configuration model}, which considers the probability space whose elements may be generated by the following simple randomized procedure. Start with $nd$ points in $n$ buckets labelled $1,\ldots,n$, with $d$ points in each bucket, and choose uniformly at random (u.a.r.) a \emph{pairing} $P=a_1,\ldots,a_{dn/2}$ of the points such that each $a_i$ is an unordered pair of points, and each point is in precisely one pair $a_i$. As usual $\mathcal{P}_{n,d}$ denotes the probability space of such pairings. By collapsing each bucket into a single vertex, we see that each pairing corresponds to a $d$-regular pseudograph  (loops and multiple edges permitted) with vertex set $\{1,\ldots,n\}$ and with an edge $\{i,j\}$ for each pair with points in buckets $i$ and $j$. A straightforward calculation shows that any two simple $d$-regular graphs (i.e.\  with no loops or multiple edges) on $n$ vertices are produced with the same probability. For fixed $d$, a crucial property is that the probability that a random pairing produces a $d$-regular graph tends to the positive constant $e^{(1-d^2)/4}$ as $n$ tends to infinity (Bender and Canfield~\cite{bender}), and so results that hold a.a.s.\ for random pairings in $\mathcal{P}_{n,d}$ must also hold a.a.s.\ for random $d$-regular graphs. 

When generating a random pairing, we may choose the pairs sequentially: the first point in a pair can be selected using any rule, as long as the second is chosen u.a.r.\ from the remaining points. We call this \emph{exposing} the pair, and this property is the \emph{independence property} of the model. The idea is to recast Algorithm~\ref{algo:totall} as if the input graph were a random regular graph that is generated while the algorithm is applied. To this end, we shall start with $P_0$, a collection of $n$ buckets with $d$ unpaired points in each bucket and with a set $D_{0}=\emptyset$. At each round $t \geq 1$, the algorithm extends a \emph{partial pairing} $P_{t+1}$ by exposing some pairs in $P_t$. It also adds vertices to $D_t$ to produce $D_{t+1}$. The \emph{degree} of a vertex (bucket) $v$ of $P_t$ is the number of unpaired points in $v$. 

\vspace{0.5cm}

\begin{algorithm}[H]
\linespread{1.2}\selectfont
\SetAlgoRefName{$1(c)$}
\KwIn{The parameters are the input.}
\KwOut{An $n$-vertex $d$-regular pseudograph $G$ and a total dominating set $D$ of $G$.}
$t=0$, $D_0=\emptyset$, $P_0 \leftarrow$ collection of $n$ buckets with $d$ points in each\;
\While{the number of vertices of degree $0$ in $P_t$ is at least $\epsilon n$}{
Choose a vertex $v_t$ u.a.r.\ among all vertices of degree $0$ in $P_t$\;
Expose a pair with a point in $v_t$. Let $u_t$ be the other vertex in the pair\;
Expose pairs for all remaining unpaired points of $u_t$\;
\uIf{$\deg_{P_t}(u_t)\neq 0$}{define $D_{t+1}=D_t \cup \{ u_t \}$\;}\Else{Expose pairs for all remaining unpaired points of $v_t$\;

\uIf{$u_t$ has a neighbor of degree 0 in $P_t$ other than $v_t$}{define $D_{t+1}=D_t \cup \{ v_t,u_t \}$\;}\Else{\uIf{$v_t$ has a neighbor $w_t$ of degree 0 in $P_t$ other than $u_t$}{expose pairs for all remaining unpaired points of $w_t$\;

define $D_{t+1}=D_t \cup \{ v_t,w_t\}$\;}\Else{choose u.a.r.\ a neighbor of $v_t'$ of $v_t$\;

expose pairs for all remaining unpaired points of $v_t'$\;

choose u.a.r.\ a neighbor $u_t'$ of $u_t$\;

expose pairs for all remaining unpaired points of $u_t'$\;

define $D_{t+1}=D_t \cup \{ v_t',u_t' \}$\;}}}
$P_{t+1}$ is the partial pairing obtained by exposing the pairs in $P_t$\; $t \leftarrow t+1$\;}
Produce $D$ from $D_t$ and $P$ from $P_t$ by exposing all of the remaining pairs in $P_t$ and adding to $D$ one neighbor of each vertex of degree 0 in $P_t$\;
\caption{with parameters $n\geq 4$, $d\geq 3$ and $\epsilon>0$.}
\label{algo:totalc}
\end{algorithm}

\vspace{0.5cm}

In particular, the edges incident to vertices that have already been deleted from the survival graph $G_t$ Algorithm~\ref{algo:totall} are precisely the pairs that have already been generated in the partial random pairing $P_t$. On the other hand, the edges of the survival graph $G_t$ correspond to the unpaired points in $P_t$. In particular, vertices of degree $i$ in $P_t$ correspond to vertices of degree $d-i$ in $G_t$.

The relevant variables associated with this heuristic will be $Q(t)=|D_t|$ and $Y_i(t)$, the number of vertices of degree $i$ in $P_t$, for $i \in \{0,\ldots,d\}$. In fact, since vertices of degree $d$ do not affect the remainder of the application of the algorithm, we ignore the variable $Y_d(t)$. We write $h_t=(P_0,\ldots, P_t)$ to denote the history of the process to time $t$ (that is, the results obtained in an actual application of the heuristic up to round $t$). The basic idea of the differential equation method is to keep track of the expected value of each variable at each round. If some technical conditions are met, a powerful result by Wormald, see for instance~\cite[Theorem 5.1]{wormald1}, implies that the actual values of the variables are a.a.s.\ close to their expected value \emph{for all} $t \in \{0,\ldots,T_C\}$. To achieve them, we shall prove that the following conditions are satisfied (we observe that some of them are stronger than what is actually needed for~\cite[Theorem 5.1]{wormald1}):
\begin{itemize}
\item[(i)]  There is an absolute constant $\beta=\beta(d)$ such that 
\[1 \leq Q(t+1)-Q(t)\leq 2 \textrm{ and }  \max_{0\leq {j}\leq d-1}|Y_{j}(t+1)-Y_{j}(t)|\leq\beta\]
for all $j \in \{0,\ldots,d-1\}$ and all $t \in \{0,\ldots,T_D\}$.

\item[(ii)]  There exist functions $f_0,f_1,\dots,f_{d-1},f_d:\mathbb{R}^{d+1}\rightarrow \mathbb{R}$ and $\lambda_1=\lambda_1(n)=o(1)$ such that, for all $0 \leq j\leq d-1$,
\[|\mathbb{E}[Y_{j}(t+1)-Y_{j}(t)|h_t]-f_{j}(t/n,Y_0(t)/n,\dots,Y_{d-1}(t)/n)|\leq\lambda_1(n)\]
and \[|\mathbb{E}[Q(t+1)-Q(t)|h_t]-f_d(t/n,Y_0(t)/n,\dots,Y_{d-1}(t)/n)|\leq\lambda_1(n)\]
for all $t< T_D$.

\item[(iii)]  The functions $f_j$ defined in (ii) are Lipschitz continuous in a domain
\[D\cap \{ (t,z_0,\dots,z_{d-1}):t\geq 0 \},\]
where $D$ is an open, connected and bounded set containing the point $(x_0,z_0,\ldots,z_{d-1})=(0,1,0,\ldots,0)$.
\end{itemize}
Roughly speaking, condition (i) tells us that the variables cannot vary substantially in a single round of the heuristic, condition (ii) tells us that the expected change in the variables (conditional on the history of the process) may be estimated with good precision, while condition (iii) tells us that these expected changes are described by well-behaved functions. If these conditions are met, Theorem~5.1~\cite{wormald1} establishes the following:
\begin{itemize}

\item[(a)] The system of differential equations associated with the functions $f_j$ has a unique solution $(z_0(x),\ldots,z_{d-1}(x),q(x))$ with initial conditions $z_0(0)=1$, $z_i(0)=0$ for $i > 0$ and $q(0)=0$. 

\item[(b)] The variables $Q(t)$ and $Y_i(t)$ are a.a.s.\ approximated throughout the process by the solutions of a system of differential equations involving the functions defined in (ii). 
More precisely, for $\lambda>\lambda_1$, there is an absolute constant $C$ such that, with probability $1-O\left(\frac{\beta}{\lambda}\exp\left( -\frac{n\lambda^3}{\beta^3} \right)\right)$, we have
\begin{equation}\label{conc_nick}
Y_{j}(t)/n=z_{j}(t/n)+O(\lambda), \quad Q(t)/n=q(t/n)+O(\lambda)
\end{equation}
for all $j$ and all $0\leq t\leq\sigma n$, where $\sigma=\sigma(n)$ is the supremum of all $x$ such that the solution to the system of differential equations may be extended up to distance at most $C\lambda$ from the boundary of $D$.
\end{itemize}

\section{Proving our main results}\label{sec:3}

In this section, we argue that the conditions (i), (ii) and (iii) described in the previous section are satisfied for our heuristic. In particular, we compute the functions $f_0,\ldots,f_{d-1},f_d$  that give rise to the system of differential equations mentioned in the statement of Theorems~\ref{thm_main} and~\ref{teo:52}.

To get started, fix integers $n>d\geq 3$. We shall assume that $n$ is sufficiently large. Assume that the process described in Algorithm~\ref{algo:totalc} runs for $T=T(n)$ rounds, let $h_t=(P_0,\ldots,P_T)$ denote the history of the process and let $D_t$ be the set produced up to round $t$.

We first note that $\beta=4d$ works for (i), as at most $4d-3$ pairs are exposed at each round, involving at most $4d-2$ vertices. To verify (ii), we need to compute $\mathbb{E}[X(t+1)-X(t)|h_t]$ for each relevant variable $X$. In fact, the independence property of the pairing process ensures that the conditional expectations in this process may be computed based on $P_t$, rather than on the full history $h_t$.

For $k \in \{0,\ldots,d-1\}$, let
\begin{equation}
S_k(t)=\sum_{i=k}^{d-1}(d-i)Y_i(t).
\end{equation}
Note that $S_k(t)$ denotes the number of unpaired points in vertices of degree at least $k$ in $P_t$. Also define
\begin{equation}
\Delta_{j,d}^{k+}(t)= \left\{ \begin{array}{ll}
(d-j+1)Y_{j-1}(t)-(d-j)Y_j(t) & \textrm{if $j>k$,}\\
-(d-j)Y_j(t) & \textrm{if $j=k$,}\\
0 & \textrm{if $j<k$.}\\
\end{array} \right.
\end{equation}

Clearly, the probabilities that a random point is chosen in a bucket of degree $i$ and in a bucket of degree at least $i$ in $P_t$ are equal to
$Y_i(t)/S_0(t)$ and $S_i(t)/S_{0}(t)$, respectively. More generally, the probability that a random point is chosen in a bucket of degree $i$ given that it lies in a vertex of degree at least $j$ (assuming that $S_j(t)>0$) is equal to $\delta_{i \geq j} Y_i(t)/S_j(t)$, where $\delta_A$ is equal to 1 if $A$ holds and is equal to 0 if $A$ does not hold. Of course, several points are paired in the same round of the algorithm, and the probability of each new choice will be affected by previous choices. Recall that at most $4d-3$ pairs are exposed in each step. So, as long as the number of unpaired points counted by $S_j(t)$ is at least $\xi n$ for some constant $\xi$, these probabilities can vary at most $O(1/n)$\footnote{All asymptotics in this paper is with respect of $n$.} within the same round, which turns out to be negligible in our computations. (Recall that $S_0(t) \geq d Y_0(t) \geq \varepsilon n$ throughout the algorithm.) Because of this, when computing expected changes in our variables, we will pretend that these random choices are independent. For the same reason, the probability that we produce loops or multiple edges in any particular step is negligible and will be absorbed by the error term.

From this, we deduce that the expected change $\Delta Y_j$ on the number of vertices of degree $j$ when $i$ points in a vertex are paired to points in vertices of degree at least $k$ (assuming that $S_k(t) \geq \xi n$ for some $\xi>0$) is given by
\begin{equation}\label{eq_atleastk}
\alpha_{i,k}^{(j)}(t) =\dfrac{i \Delta_{j,d}^{k+}(t)}{S_k(t)}+o(1).
\end{equation}
Next we use the simple fact that, for any event $A$ and any random variable $X$ we have 
$\mathbb{E}[X|\overline{A}]\mathbb{P}[\overline{A}]=\mathbb{E}[X]-\mathbb{E}[X|A]\mathbb{P}[A]$. Then the expected change $\Delta Y_j$ on the number of vertices of degree $j$ when $i$ points are paired, and at least one point is paired to a vertex of degree at most $k-1$, is equal to
\begin{equation}\label{eq_atmostk}
\beta_{i,k}^{(j)}(t) =\dfrac{i \dfrac{\Delta_{j,d}^{0+}(t)}{S_0(t)}-i\dfrac{\Delta_{j,d}^{k+}(t)}{S_k(t)}\left( \dfrac{S_k(t)}{S_0(t)}\right)^{i}}{1-\left( \dfrac{S_k(t)}{S_0(t)}\right)^{i}}+o(1).
\end{equation}

With this, we may compute the expected effect on $\Delta_j$ of selecting $u_t$ with the properties in each line of the algorithm. For line 6, we need to have $\deg_{P_t}(u_t)>0$. The expected change on $\Delta_j$ is 
\begin{equation}\label{effect1}
\frac{S_1(t)}{S_0(t)} \left(-\delta_{j=0}+\delta_{j=1}\right)+\sum_{i=1}^{d-1} \frac{(d-i)Y_i(t)}{S_0(t)}\left(-\delta_{j=i}+\alpha_{d-i-1,0}^{(j)}(t) \right).
\end{equation}
The expressions $\frac{S_1(t)}{S_0(t)}$ and $ \frac{(d-i)Y_i(t)}{S_0(t)}$ are the probabilities that $u_t$ has degree at least 1 and degree $i$, respectively. The first term in the sum is due to the change in the degree of $v_t$ and the terms in the sum are due to the deletion of $u_t$ and to the changes in the degrees of the neighbors of $u_t$ that have been exposed when the remaining points in $u_t$ have been paired. 

The expected effect of choosing $u_t$ that satisfies line 10 (in this case, we know that $u_t$ has degree 0 and has a neighbor of degree 0 other than $v_t$) is
\begin{equation}\label{effect2}
p_2(t) \left(-2\delta_{j=0}+\alpha_{d-1,0}^{(j)}(t)+\beta_{d-1,1}^{(j)}(t) \right)+o(1).
\end{equation}
The expression is multiplied by the probability
$$p_2(t)=\frac{dY_0(t)}{S_0(t)}\left(1- \left(\frac{S_1(t)}{S_0(t)} \right)^{d-1}\right)$$
that $u_t$ has degree 0 and at least one of its neighbors other than $v_t$ has degree 0. The first term in the sum is due to the deletion of $v_t$ and $u_t$ and the other two terms are due to the changes in the degrees of the neighbors of $v_t$ and $u_t$, respectively.

Assume that we choose $u_t$ and $v_t$ as in line 13, that is, $v_t$ and $u_t$ have degree 0 in $P_t$, $v_t$ has a neighbor $w_t$ of degree 0 other than $u_t$, but $v_t$ is the single neighbor of $u_t$ of degree 0. The expected effect on $\Delta_j$ is
\begin{equation}\label{effect3}
p_3(t) \left(-2\delta_{j=0}+\alpha_{d-1,1}^{(j)}(t)+\beta_{d-1,1}^{(j)}(t)+\alpha_{d-1,0}^{(j)}(t) -\delta_{j=1} \right)+o(1).
\end{equation}
The first term in the sum is due to the deletion of $v_t$ and $u_t$. The second term comes from degrees in the neighborhood of $u_t$, the third from the neighborhood of $v_t$ and the fourth from the neighborhood of $w_t$. Note that the third term counts turning $w_t$ from a vertex of degree 0 into a vertex of degree 1, so that the term $-\delta_{j=1}$ must be added to account for the deletion of $w_t$. Everything is multiplied by the probability 
$$p_3(t)=\frac{dY_0(t)}{S_0(t)}  \left(\frac{S_1(t)}{S_0(t)} \right)^{d-1} \left(1- \left(\frac{S_1(t)}{S_0(t)} \right)^{d-1}\right)$$ 
that $u_t$ and $v_t$ satisfy the conditions of this case.

Finally, assume that we are in line 16. So $v_t$ and $u_t$ have degree 0 in $P_t$, but their remaining neighbors have degree at least 1. The expected effect on $\Delta_j$ is
\begin{equation}\label{effect4}
p_4(t) \left(-2\delta_{j=0}+
2\alpha_{d-1,1}^{(j)}(t)+2\sum_{m=1}^{d-1}\dfrac{(d-m)Y_m(t)}{S_1(t)}\left[ -\delta_{m+1,j}+\alpha_{d-m-1,0}^{(j)}(t) \right]\right)+o(1).
\end{equation}
The first term in the sum accounts for the deletion of $v_t$ and $u_t$. The second term for their neighborhoods. The sums refer to the neighborhoods of $u_t'$ and $v_t'$ and take their degrees into account (note that the probability of having degree $i$ is conditional upon having degree at least 1). Observe that the terms $-\delta_{m+1,j}$ appear because the change in the degrees of $u_t'$ and $v_t'$ is counted when considering the neighborhoods of $u_t$ and $v_t$. This is multiplied by the probability 
$$p_4(t)=\frac{dY_0(t)}{S_0(t)}  \left(\frac{S_1(t)}{S_0(t)} \right)^{2(d-1)}$$
that $u_t$ and $v_t$ satisfy the conditions in this case.

We may now sum the equations \eqref{effect1}-\eqref{effect4} above to write the conditional expectation $\mathbb{E}[Y_j(t+1)-Y_j(t)|G_t]$ in the form
$$E[Y_j(t+1)-Y_j(t)|G_t]=f_j(t/n,Y_0(t)/n,Y_1(t)/n,\dots,Y_{d-1}(t)/n)+o(1),\quad 0\leq j\leq d-1.$$
where the functions $f_j$ are rational functions on $d+1$ variables that are well-defined whenever $S_0(t)$ is positive, which is always the case if $Y_0(t)$ is positive. Explicit expressions for the functions $f_j$ are in Appendix~\ref{Ap1}.

Next, consider the the function $Q(t)=|D(t)|$ that keeps track of the size of the total dominating set. The algorithms adds a single vertex to the total dominating set in round $t$ if $\deg_{G_t}(u_t)>0$, while two vertices are added to this set otherwise. As a consequence,
\begin{eqnarray*}
E[Q(t+1)-Q(t)|G_t] &=& \dfrac{S_1(t)}{S_0(t)}+2 \cdot \dfrac{dY_0(t)}{S_0(t)}+o(1) \\
&=& f_d(t/n,Y_0(t)/n,\dots,Y_{d-1}(t)/n)+o(1)
\end{eqnarray*}
Rewriting the quantities involved in the recurrence relations in terms of the normalized variables
\[x=t/n,\quad y_i(x)=Y_i(xn)/n, \quad q(x)=Q(xn)/n,\]
and letting $n\rightarrow\infty$, we may view this system recurrence relations as a discretization of the following system of differential equations and initial conditions:
\begin{eqnarray} \label{eq:53}
\left\{ \begin{array}{lll}
z_j'(x) &=& f_j(x,z_0,z_1,\dots,z_{d-1}) \quad \textrm{for all $0\leq j\leq d-1$}\\
q'(x) &=& f_d(x,z_0,z_1,\dots,z_{d-1})\\
z_0(0) &=& 1, z_j(0)=0 \textrm{ for $1\leq j\leq d-1$},~q(0) = 0.\\
\end{array} \right.
\end{eqnarray}
At this point, we have found the functions $f_j$ that verify (ii) with $\lambda_1=o(1)$.

Next we define a domain $D \subseteq\mathbb{R}^{d+1}$ for which (iii) is satisfied. For $\varepsilon>0$, let $D_\varepsilon$ contain all tuples $(x,z_0,z_1,\dots,z_{d-1})\in \mathbb{R}^{d+1}$ such that $-\varepsilon<x<1$, $\varepsilon<z_0<1+\varepsilon$  and $-\varepsilon/(d-j)<z_j<1+\varepsilon$ for all $1 \leq j \leq d-1$. In particular, for any $x \in D_{\varepsilon}$, $s_0(x) \geq \varepsilon$.
\begin{proposition}
For any $\varepsilon >0$ and $j \in \{0,\ldots,d\}$, the function $f_j$ is Lipschitz continuous in $D_\varepsilon$. 
\end{proposition}

\begin{proof}
Note that each $f_j$ is a rational function of the form $p_j/r_j$, where $p_j(x,z_0\dots,z_{d-1})$ and $r_j(x,z_0\dots,z_{d-1})$ are multivariate polynomials on $d+1$ variables such that $r_j$ does not contain roots in the closure of $D_\varepsilon$. In particular, the functions $f_j$ are continuous and have continuous derivatives in the closure of $D_\varepsilon$, and therefore are Lipschitz continuous in $D_\varepsilon$. 
\end{proof}

Since conditions (i), (ii) and (iii) have been verified, we may apply the differential equation method of Section~\ref{sec:2_1} to derive (a) and (b). That is, the system of differential equations associated with the functions $f_j$ has a unique solution $(z_0(x),\ldots,z_{d-1}(x),q(x))$ with initial conditions $z_0(0)=1$, $z_i(0)=0$ for $i > 0$ and $q(0)=0$ that may be extended to a value $\sigma$ arbitrarily close to the boundary of $D_\varepsilon$. Moreover, for $\lambda(n)=\max\{ n^{-1/4},2\lambda_1(n) \}$, the equations
\begin{equation*}
Y_{j}(t)/n=z_{j}(t/n)+O(\lambda), \quad Q(t)/n=q(t/n)+O(\lambda)
\end{equation*}
hold with high probability for all $t$ up to $\sigma n$.

We still need to prove that step $\sigma n$ occurs in a region where $z_0$ is small and that $z_j(x) \geq 0$ for all $x \in [0,\sigma]$. Intuitively, this means that the process ends because the number of vertices of degree 0 is getting to small. This allows us to carry out the analysis up to a point where almost all vertices of the input graph have been totally dominated. To prove this, we shall establish properties of the solutions to the system of differential equations.

The results below ensure that the solutions $z_j(x)$ and $q(x)$ lie within the interval $[0,1]$ for all values of $x \geq 0$ such that $z_0(x)>0$. This implies that the reason why the vector of solutions approaches the boundary of the closure of $D_\varepsilon$ is that $z_0(x)$ approaches 0.
\begin{proposition}\label{prop_begin}
There exists $\delta>0$ such that, for all $x\in(0,\delta]$ and $0\leq j\leq d-1$, we have $z_j(x)>0$.\label{prop:50}
\end{proposition}
\begin{proof}
Since $z_0(0)=1$ and $z_0(x)$ is differentiable in $x=0$, there is $\delta_0>0$ such that $z_0(x)>0$ for all $x\in[0,\delta_0]$.

Consider the differential equations involving $z_j'$ for $1\leq j\leq d-1$. To obtain the desired result, we shall prove that the nonzero derivative of smallest order of each of $z_j$ at the point $x=0$ must be positive.
\begin{lemma}
For $1\leq j\leq d-1$, we have $z_j^{(k)}(0)=0$ for $1\leq k\leq j-1$ and $z_j^{(j)}(0)>0$.
\end{lemma}
\begin{proof}
We prove this by induction on $j\geq 1$. The base of induction follows from
$$z_1'(0)=\dfrac{2(d-1)\Delta_{1,d}^{0+}(0)}{d n}=\frac{2(d-1)\phi_{1,d}^{0+}(0)}{d}=2(d-1)>1,$$
as $p_2(0)=1$ and the terms \eqref{effect1}, \eqref{effect3} and \eqref{effect4} are zero at $x=0$.

Next, for every $j\in\{ 2,\dots,d-1 \}$, each term in the sum that produces $f_j(x,z_0,z_1,\dots,z_{d-1})$ contains a factor $z_j(x)$ or a factor $z_{j-1}(x)$, so that the differential equation may be rewritten in the form
\begin{equation}\label{eq:52}
z_j'(x)=z_{j-1}(x)u_{1,j}(x)-z_{j}(x)u_{2,j}(x),
\end{equation}
where $u_{1,j}(x)$ and $u_{2,j}(x)$ are the rational funtions that result from this rearrangement. By induction, it is easy to see that $z_j$ is of class $C^\infty$ for all points $x \in D_\varepsilon$.

For the step of induction, we differentiate $(k-1)$ times both sides of equation~\eqref{eq:52} and use the induction hypothesis $z_{j-1}(0)=z_{j-1}'(0)=z_{j-1}''(0)=\dots=z_{j-1}^{(j-2)}(0)=0$ and $z_{j-1}^{(j-1)}(0)>0$. For $1\leq k\leq j-1$, we obtain
\begin{eqnarray}\label{tapsi2}
z_j^{(k)}(0) &=& \sum_{m=0}^{k-1}\binom{k-1}{m}\left( z_{j-1}^{(m)}(0)u_{1,j}^{(k-m)}(0)-z_{j}^{(m)}(0)u_{2,j}^{(k-m)}(0)\right) \nonumber \\
&=& -\sum_{m=0}^{k-1}\binom{k-1}{m} z_{j}^{(m)}(0)u_{2,j}^{(k-m)}(0).
\end{eqnarray}
For $k=1$, this implies that $z_j'(0)=0$, since $z_j(0)=0$.  Now, as $z_j(0)=0$ and $z_j'(0)=0$, equation~\eqref{tapsi2} for $k=2$ implies that $z_j''(0)=0$. This argument may be repeated to derive $z_{j}(0)=z_{j}'(0)=z_{j}''(0)=\dots=z_{j}^{(j-1)}(0)=0$. It remains to prove that $z_j^{(j)}(0)>0$, but this follows from
\begin{eqnarray*}
z_j^{(j)}(0) &=& \sum_{m=0}^{j-1}\binom{j-1}{m}\left( z_{j-1}^{(m)}(0)u_{1,j}^{(j-m)}(0)-z_{j}^{(m)}(0)u_{2,j}^{(j-m)}(0)\right)\\
&=& z_{j-1}^{(j-1)}(0)u_{1,j}(0) = \dfrac{2(d-1)(d-j+1)z_{j-1}^{(j-1)}(0)}{d} >0.
\end{eqnarray*}
This concludes the proof.
\end{proof}

As a consequence, for $1\leq j\leq d-1$, the Taylor expansion of order $j$ of $z_j(x)$ centered at $x=0$ satisfies $$z_j(x)=\frac{z_j^{(j)}(0)}{j!}x^j+r_j(x),$$
where $r_j(x)$ is such that
$$\lim_{x\rightarrow 0^{+}}\frac{r_j(x)}{x^j}=0.$$
Therefore there is $\delta_j>0$ such that, for all $x\in(0,\delta_j]$, we have $z_j(x)>0$. Setting $\delta=\min\{ \delta_j:0\leq j\leq d-1 \}$ concludes the proof of Proposition~\ref{prop_begin}.
\end{proof}

\begin{theorem} \label{teo:supz}
Given solutions $(z_0,z_1,\ldots,z_{d-1},q)$ to \eqref{eq:53}, let 
$$x^*=\sup\{ \theta | \textrm{$z_j(x)>0$ for all $x\in (0,\theta)$ and $j \in \{0,\ldots,d-1\}$} \}.$$ 
Then
\begin{itemize}
\item[(i)] $z_j(x)\leq 1$ for all $0\leq j\leq d-1$ and $x\in [0,x^*)$;
\item[(ii)] $x^*\in (0,1]$;
\item[(iii)] $z_0(x^*)=0$.
\end{itemize}
\end{theorem}

\begin{proof}
The first two items follow immediately from the fact that, for $F(x)=z_0(x)+z_1(x)+\dots+z_{d-1}(x)$, we have $F(0)=1$ and $F'(x) \leq -1$ for all $x \in [0,x^*)$.

To prove (iii), we claim that, if $z_j(x^*)=0$ for some $j$ such that $1 \leq j\leq d-1$, then $z_{j-1}(x^*)=0$. Iterating this argument leads to  $z_0(x^*)=0$. To establish our claim, suppose for a contradiction that $z_{j}(x^*)=0$, but $z_{j-1}(x^*)\neq 0$, i.e.\, $z_{j-1}(x^*)>0$. Looking at our expression for $f_j$, one may easily see that the terms $b'^j_i(x^*)$ and $e^j_i(x^*)$ must be strictly positive, which implies that $z'_j(x^*)>0$. Since $z_{j}(x^*)=0$, we would find $\delta>0$ such that $z_j(x^*-\delta)<0$, which contradicts our choice of $x^*$. 
\end{proof}

With these results, we may now conclude the proof of Theorem~\ref{teo:52}. Given $\delta>0$, we choose $0<\varepsilon<\delta/2$ and define
$$\hat{x}=\sup\{ \theta | \textrm{$z_0(x)\geq \varepsilon$ for all $x\in (0,\theta)$} \},$$
By the differential equation method, Algorithm~\ref{algo:totalc} asymptotically almost surely produces a random $d$-regular pseudograph that contains a total dominating set $D$ of size 
$$Q(\hat{x}n)+ Y_0(\hat{x}n)\leq q(\hat{x})n+\frac{\delta n}{2} + \varepsilon n + \leq q(x^*)n+\delta n,$$
as required. We are using that $q(\hat{x})\leq q(x^*)$.

\vspace{0.5cm}
\noindent {\bf Acknowledgments:}
C. Hoppen acknowledges the support of CNPq (Proj.~308054/2018-0), Conselho Nacional de Desenvolvimento Cient\'{i}fico e Tecnol\'{o}gico. G. Mansan thanks Coordena\c{c}\~{a}o de Aperfei\c{c}oamento de Pessoal de N\'{i}vel Superior (CAPES) for their support.

\appendix

\section{The functions $f_j$} \label{Ap1}

In Section~\ref{sec:3}, we defined the IVP~\eqref{eq:53} by looking at the expected change in the value of the variables $Y_j$ when performing a single round of the algorithm. In this section, we write the full algebraic expressions for completeness. 

By replacing the terms $\alpha_{i,k}^{(j)}$ and $\beta_{i,k}^{(j)}$ by their original expressions, the functions $f_j$ may be written as
$$f_j(x,z_0,\dots,z_{d-1})=\sum_{i=1}^{d-1}b_i(x)b'^j_i(x)+\sum_{m=2}^{4}p_m(x)e^j_m(x),$$
where the first term comes from~\eqref{effect1} and the three terms in the sum come from~\eqref{effect2}-\eqref{effect4}. Here,
\begin{eqnarray*}
b_i(x) &=& \dfrac{(d-i)z_i(x)}{s_0(x)},\\
b'^j_i(x) &=& -\delta_{j=0}+\delta_{j=1}-\delta_{j=i}+(d-i-1)\dfrac{\phi_{j,d}^{0+}(x)}{s_0(x)},
\end{eqnarray*}
\begin{eqnarray*}
p_2(x) &=& \dfrac{dz_0(x)}{s_0(x)}\left(1- \left(\dfrac{s_1(x)}{s_0(x)}\right)^{d-1}\right),\\
e_2^j(x) &=& -2\delta_{j=0}+(d-1)\dfrac{\phi_{j,d}^{0+}(x)}{s_0(x)}+h_{j,d}(x),\\
p_3(x) &=& \dfrac{dz_0(x)}{s_0(x)}\left(\dfrac{s_1(x)}{s_0(x)}\right)^{d-1}\left(1- \left(\dfrac{s_1(x)}{s_0(x)}\right)^{d-1}\right),\\
e_3^j(x) &=&-2\delta_{j=0}+(d-1)\dfrac{\phi_{j,d}^{1+}(x)}{s_1(x)}+h_{j,d}(x)-\delta_{j=1}+(d-1)\dfrac{\phi_{j,d}^{0+}(x)}{s_0(t)},\\
p_4(x) &=& \dfrac{dz_0(x)}{s_0(x)}\left(\dfrac{s_1(x)}{s_0(x)}\right)^{d-1}\left(\dfrac{s_1(x)}{s_0(x)}\right)^{d-1},\\
e_4^j(x) &=& -2\delta_{j=0}+2(d-2)\dfrac{\phi_{j,d}^{1+}(x)}{s_1(x)}+2\sum_{m=1}^{d-1}\dfrac{(d-m)z_m(x)}{s_1(x)}\left[ -\delta_{m=j}+(d-m-1)\dfrac{\phi_{j,d}^{0+}(x)}{s_0(x)} \right],
\end{eqnarray*}
\begin{eqnarray*}
h_{j,d}(x) &=& \dfrac{(d-1)\dfrac{\phi_{j,d}^{0+}(x)}{s_0(x)}-(d-1)\dfrac{\phi_{j,d}^{1+}(x)}{s_1(x)}\left( \dfrac{s_1(x)}{s_0(x)}\right)^{d-1}}{1-\left( \dfrac{s_1(x)}{s_0(x)}\right)^{d-1}},\\
\phi_{j,d}^{k+}(x) &=& \left\{ \begin{array}{ll}
(d-j+1)z_{j-1}(x)-(d-j)z_j(x) & \textrm{if $j>k$}\\
-(d-j)z_j(x) & \textrm{if $j=k$}\\
0 & \textrm{if $j<k$}\\
\end{array} \right. ,\\
s_k(x) &=& \sum_{i=k}^{d-1}(d-i)z_i(x).
\end{eqnarray*}

In order to see that the singularities of the functions $f_j$ are precisely the points such that $s_0(x)=0$, note that
\begin{eqnarray*}
\lefteqn{p_2(x)e^j_2(x)=\dfrac{dz_0(x)}{s_0(x)}\left(1- \left(\dfrac{s_1(x)}{s_0(x)}\right)^{d-1}\right) \left(-\delta_{j,0}-\delta_{j,0}+(d-1)\dfrac{\phi_{j,d}^{0+}(x)}{s_0(x)}\right)}\\
& & +\dfrac{dz_0(x)}{s_0(x)}\left((d-1)\dfrac{\phi_{j,d}^{0+}(x)}{s_0(x)}-(d-1)\dfrac{\phi_{j,d}^{1+}(x)}{s_0(x)}\left( \dfrac{s_1(x)}{s_0(x)}\right)^{d-2}\right).
\end{eqnarray*}
Moreover,
\begin{eqnarray*}
\lefteqn{p_3(x)e^j_3(x)=\dfrac{dz_0(x)}{s_0(x)}\left(\dfrac{s_1(x)}{s_0(x)}\right)^{d-1}\left( (d-1)\dfrac{\phi_{j,d}^{0+}(x)}{s_0(x)}-(d-1)\dfrac{\phi_{j,d}^{1+}(x)}{s_0(x)}\left( \dfrac{s_1(x)}{s_0(x)}\right)^{d-2} \right)} \\
& & +\dfrac{dz_0(x)}{s_0(x)}\left(\dfrac{s_1(x)}{s_0(x)}\right)^{d-2}\left(1- \left(\dfrac{s_1(x)}{s_0(x)}\right)^{d-1}\right)\left( \dfrac{\phi_{j,d}^{1+}(x)}{s_0(x)} \right)\\
& & +\dfrac{dz_0(x)}{s_0(x)}\left(\dfrac{s_1(x)}{s_0(x)}\right)^{d-1}\left(1- \left(\dfrac{s_1(x)}{s_0(x)}\right)^{d-1}\right)\left( -\delta_{j,0}-\delta_{j,0}-\delta_{j,1}+(d-1)\dfrac{\Delta_{j}^{0+}(x)}{s_0(t)} \right),
\end{eqnarray*}
and
\begin{eqnarray*}
\lefteqn{p_4(x)e^j_4(x)=\dfrac{dz_0(x)}{s_0(x)}\left(\dfrac{s_1(x)}{s_0(x)}\right)^{d-1}\left(\dfrac{s_1(x)}{s_0(x)}\right)^{d-1}\left( -2\delta_{j,0}\right)} \\
& & +\dfrac{dz_0(x)}{s_0(x)}\left(\dfrac{s_1(x)}{s_0(x)}\right)^{2d-3} \left(2\sum_{m=1}^{d-1}\dfrac{(d-m)z_m(x)}{s_0(x)}\left[ -\delta_{m,j}+(d-m-1)\dfrac{\phi_{j,d}^{0+}(x)}{s_0(x)} \right] \right),\\
& & +\dfrac{dz_0(x)}{s_0(x)}\left(\dfrac{s_1(x)}{s_0(x)}\right)^{2d-3} \left( 2(d-2)\dfrac{\phi_{j,d}^{1+}(x)}{s_0(x)} \right).
\end{eqnarray*}
\end{document}